\documentclass[]{article}

\usepackage{amsthm}
\usepackage{amssymb}
\usepackage{amsmath}

\newtheorem{theorem}{Teorema}[section]

\newtheorem{thm}[theorem]{Theorem}
\newtheorem{corollary}[theorem]{Corolary}
\newtheorem{lemma}[theorem]{Lemma}
\newtheorem{prop}[theorem]{Proposition}
\newtheorem{definition}[theorem]{Definition}

\newtheorem{example}[theorem]{Example}

\newcommand{\R}{\mathbb{R}}

\newcommand{\U}{\mathcal{U}}

\newcommand{\Bo}{\mathcal{B}}
\newcommand{\Co}{\mathcal{C}}
\newcommand{\Po}{\mathcal{P}}
\newcommand{\agl}{\operatornamewithlimits{agl}}
\newcommand{\Ne}{\mathcal{N}}

\title{The Baire property in uniform spaces: a survey}
\author{Francisco Javier García Pacheco\\Álvaro García Zambrano\\Fernando Rambla Barreno\\
\\
\small University of Cádiz, Cádiz, Spain}
\date{May 2026}

\begin{document}

\maketitle

\begin{abstract}
	\noindent
	The Baire category theorem states that every complete pseudometric space is a Baire space. There are some results in metric spaces which have their analogue in uniform spaces, however this is not one of them. Nonetheless, since the Baire property is always desirable, we decide to explore some conditions, such as countable compacity, pseudocompacity and pseudocompletenes, and see in which circumstances a general complete uniform space satisfies the Baire property.\\
	
	
\end{abstract}

\section{Introduction}

The Baire category theorem (BCT) is a classical result in Topology and Functional Analysis which provides a really powerful tool. It is used to prove the open mapping theorem (and thus the closed graph theorem) and that if a Banach space has infinite dimension, it must be uncountable, among many other results. This theorem indicates two sets of conditions under which a given topological space $X$ is a Baire space, and one of them is for $X$ to be a complete metric (even pseudometric) space. As we know, uniform spaces are a natural generalization of metric spaces. A variety of classical properties and results from metric spaces can be extended to uniform spaces, being completeness one of these properties. This suggests the following question: is a complete uniform space a Baire space? The answer to this is no, for we can find examples of complete uniform spaces which are not Baire. In \cite[39B.2]{willard} it is stated that every paracompact space is completely uniformizable, and since any metric (not necessarily complete) space is paracompact we can find instances, such as $\mathbb{Q}$, that can be given a complete uniform structure even though they are clearly not Baire.\\

Indeed, Willard says in \cite[39B]{willard} that ``a completely uniformizable metric space need not to be completely metrizable, nor can we prove a Baire theorem (25.4) for completely uniformizable spaces". Joshi too states in \cite[pg 363]{joshi} that ``completeness of uniform spaces is a straightforward generalization of the same notion for metric spaces. Many results about the latter have expected analogues for the former. Unfortunately, there seems to be no analogue of the Baire category theorem for complete uniform spaces".\\

In this article, we investigate some conditions that, when added to completion in a uniform space, yield the Baire property.

\section{Uniform spaces}

Uniform spaces are a structure introduced by André Weil which naturally generalizes metric/pseudometric spaces, and are defined in terms of filters as follows \cite{willard}.

\begin{definition}
	\label{defuniformsapce}
	Let $X$ be a nonempty set, we say that $\U\subset\mathcal{P}(X\times X)$, whose elements are called entourages, is a uniform structure (or a uniformity) over $X$ if:
	\begin{enumerate}
		\item[1.] $\U$ is a filter.
		
		
		\item[2.] For each $U\in\U$ the diagonal $\Delta$ verifies $\Delta\subset U$.
		
		\item[3.] If $U\in\U$ then $U^{-1}=\{(y,x):(x,y)\in U\}\in\U$. An entourage $U$ is said to be symmetric if $U=U^{-1}$.
		
		\item[4.] For $U,V\in\U$ we define $U\circ V=\{(x,y):\exists z\in X\text{ such that }(x,z)\in U,(z,y)\in V\}$.  Then for each $U\in\U$ there exists $V\in\U$ such that $V\circ V\subset U$, where $V\circ V=\{(x,y):\exists z\in X\text{ such that }(x,z),(z,y)\in V\}$. We denote $V^n=\underset{n}{\underbrace{V\circ\ldots\circ V}}$.
	\end{enumerate}
\end{definition}

We say that $\mathcal{B}\subset\mathcal{U}$ is a base of uniformity or a fundamental system of entourages if $\mathcal{B}$ is a prefilter/filter base of $\U$.\\

Given an entourage $U\in\U$, $x\in X$ and $A\subset X$, we denote $U[x]=\{y\in X:(x,y)\in U\}$ and $U[A]=\{y\in X:(x,y)\in U\text{ for some }x\in A\}$.\\

\begin{example}
	\label{example}
	Let $(X,d)$ a pseudometric space. We may construct a uniform structure as follows: for each $\varepsilon>0$ we consider $U_\varepsilon=\{(x,y)\in X\times X:d(x,y)<\varepsilon\}$. It is easy to check that the family $\mathcal{U}=\{U_\varepsilon:\varepsilon>0\}$ verifies all four conditions from Definition~\ref{defuniformsapce}
\end{example}

Any uniform structure $\U$ induces a topology $\mathcal{T_U}$ in the set $X$, which receives the name of uniform topology, and is defined as follows: a set $A\subset X$ is open if and only if for each $x\in A$ there exists $U\in\U$ such that $U[x]\subset A$. A topological space is said to be uniformizable if it admits a uniform structure compatible with its topology. Let us make some remarks about this uniform topology.\\

First of all, it is well known that any topological space is uniformizable if and only if it is completely regular \cite[38.2]{willard}.\\

Secondly, we may define the notion of a separated (or Hausdorff) uniform space as follows.

\begin{definition}
	Let $(X,\U)$ be a uniform space. It is said to be separated if $\bigcap_{U\in\U}U=\Delta$.
\end{definition}

The following result is easy to prove \cite[Chapter II I.2.3]{Bourbaki}.

\begin{prop}
	Let $(X,\U)$ be a uniform space. Then $(X,\U)$ is separated if and only if $(X,\mathcal{T_U})$ is Hausdorff.
\end{prop}

We define now the notion of total boundedness, which keeps a deep connection with compactness.

\begin{definition}
	Let $(X,\U)$ be a uniform space. A subset $S\subset X$ is said to be totally bounded if for each entourage $U\in\U$ there is a finite covering $C$ of $S$ such that $T\times T\subset U$ for each $T\in C$.
\end{definition}

Moving on, the notion of completeness can be easily defined in uniform spaces by making an adequate generalization of the notion of Cauchy sequences.

\begin{definition}
	Let $(X,\U)$ be a uniform space. A net $\{x_\lambda\}_\lambda\subset X$ is said to be Cauchy if for each $U\in\U$ there is some $\lambda_0$ such that $(x_{\lambda_1},x_{\lambda_2})\in U$ whenever $\lambda_1,\lambda_2\geq\lambda_0$. 
\end{definition}

Before continuing with Cauchy nets, let us briefly stop by Cauchy prefilters.

\begin{definition}
	Let $X$ be a uniform space. A Cauchy prefilter in $X$ is a prefilter $\Bo\subseteq \Po(X)$ such that the generated filter $\mathcal{J}(\{B\times B: B\in \Bo\})$ contains the uniformity of $X$. 
\end{definition}

Under the settings of the previous definition, $\Bo$ is a Cauchy prefilter if and only if for every entourage $U\subseteq  X\times X$ there exists $B\in\Bo$ with $B\times B\subseteq U$. This is known as the topological/analytical version of the Cauchy prefilter notion. Notice that if $\Bo,\Co$ are prefilters in a uniform space $X$ such that $\Bo\subseteq\Co$ and $\Bo$ is Cauchy, then $\Co$ is Cauchy as well.

For the following technical lemma, we remind the reader that, in a topological space $X$, if $\Bo$ is a prefilter in $X$, then $\lim(\Bo):=\left\{x\in X: \Ne_x\subseteq \mathcal{J}(\Bo)\right\}$ and $\agl(\Bo) := \bigcap_{A\in \Bo} \overline{A}$, where $\Ne_x$ is the filter of all neighborhoods of $x$ and  $\mathcal{J}(\Bo)$ is the filter generated by $\Bo$.

\begin{lemma}\label{Cc}
	Let $X$ be a uniform space. A prefilter  $\Bo$ in $X$ is   convergent if and only if $\Bo$ is a Cauchy prefilter and $\agl(\Bo)\neq \varnothing$. In this situation, $\lim(\Bo)=\agl(\Bo)$.
\end{lemma}

A net $(x_{\lambda})_{\lambda\in\Lambda}$ in a uniform space $X$ is Cauchy if and only if  the prefilter in $X$ given by  $\Bo_{(x_\lambda)_{\lambda\in\Lambda}}:=\left\{\{x_\lambda:\lambda\geq \gamma\}:\gamma\in\Lambda\right\}$ is a Cauchy prefilter. As a direct consequence of Lemma \ref{Cc}, we expect to have the following \cite[39.2]{willard}.

\begin{thm}
	A net is convergent if and only if it  is Cauchy and has nonempty agglomeration.
\end{thm}

\begin{definition}
	A uniform space $(X,\U)$ is said to be complete if every Cauchy prefilter is convergent.
\end{definition}

\begin{theorem}\label{Caus3}
	A uniform space $X$ is complete if and only if every Cauchy net in $X$ is convergent. 
\end{theorem}

By making use of the total boundedness and the completeness we can state the following \cite[39.9]{willard}.

\begin{thm}
	\label{compacto}
	Let $(X,\U)$ be a uniform space. It is totally bounded and complete if and only if $(X,\mathcal{T_U})$ is compact.
\end{thm}

To end this section, let us talk about metrization. As we have already stated, uniform spaces are a generalization of pseudometric/metric spaces, and thus, as we saw in Example~\ref{example}, any metric space admits a uniform structure, but is the converse true? Does a general uniform space admit a metric (or a pseduometric) which yields its uniform structure? The answer is no, of course. In fact, it is known when a general uniform space is metrizable \cite[6.13]{kelley}:

\begin{thm}
	\label{metrization}
	Let $(X,\U)$ be a uniform space. Then it is pseudometrizable if and only if it admits a countable base of uniformity. If it is separated, then it is metrizable.
\end{thm}

\section{The Baire category theorem}
\label{introduction}

We recall the definition of a Baire space.

\begin{definition}
	Let $X$ be a topological space. Then $X$ is said to be  a Baire space if any of the following and equivalent conditions is satisfied:
	
	\begin{itemize}
		\item The countable intersection of dense open sets in $X$ is dense too in $X$.
		
		\item The countable union of closed sets with empty interior has empty interior too.
		
		\item Every meagre subset of $X$ has empty interior.
		
		\item Every comeagre subset of $X$ is dense in $X$.
		
		\item Any open and nonempty subset of $X$ is nonmeagre.
	\end{itemize}
\end{definition}

We recall now the Baire category theorem, which can be found, for instance, in \cite[6.34]{kelley}.

\begin{thm}
	Let $X$ be a topological space. If any of the following (and nonequivalent) conditions is satisfied, then $X$ is a Baire space:
	
	\begin{itemize}
		\item (BCT1) $X$ is completely metrizable. Furthermore, $X$ can be a complete pseudometric space. 
		
		\item (BCT2) $X$ is locally compact and Hausdorff. Furthermore, $X$ can be locally compact and regular.
	\end{itemize}
	
	\begin{proof}
		\hspace{1cm}
		\begin{itemize}
			\item (BCT1) Let $\{U_n\}_n$ be a family of open sets, and let $A\subset X$ be another (nonempty) open set. Since $U_1$ is dense, we know that $U_1\cap A\neq\emptyset$, and thus there exists $x_1\in U_1\cap A$. Since $U_1\cap A$ is open, there exists $r_1>0$ such that $\overline{B}(x_1,r_1)\subset(U_1\cap A)$. Next, we take $U_2\cap\overline{B}(x_1,r_1)\subset(U_2\cap U_1\cap A)$, and because $U_2\cap B(x_1,r_1)$ is open, there exists $r_2>0$ such that $\overline{B}(x_2,r_2)\subset (U_2\cap \overline{B}(x_1,r_1))\subset (U_2\cap U_1\cap A)$. Continuing this way, we obtain two sequences $\{r_n\}_n$ (we can always choose $r_n<\frac{1}{n})$ and $\{x_n\}_n$ (here we make use of the axiom of choice).
			
			Since $r_n\rightarrow0$ and $x_n\in B(x_m,r_m)$, we infer that $\{x_n\}$ is a Cauchy sequence, and therefore there exists its limit $x\in X$. We know that $x_n\in\overline{B}(x_m,r_m)$ if $n>m$, and also $x\in\overline{B}(x_m,r_m)\subset(U_1\cap\ldots\cap U_m\cap A)$ for each $m$, so $x\in\big(\bigcap_n U_n\big)\cap A$. This shows that $\big(\bigcap_nU_n\big)\cap A$ is not empty and therefore $\bigcap_nU_n$ is dense in $X$.
			
			\item (BCT2) Let $\{U_n\}_n$ be a family of open sets, and let $A\subset X$ be another (nonempty) open set. We know that $U_1\cap A$ is open and non empty because $U_1$ is dense, so there exists $x_1\in U_1\cap A$. Since $X$ is regular, by \cite[31.1]{munkres} there exists and open set $V_1\ni x_1$, which can be chosen so that $\overline{V_1}$ is compact because $X$ is locally compact, such that $\overline{V_1}\subset (U_1\cap A)$. Besides, there exists $x_2\in (V_1\cap U_2)\subset(A\cap U_1\cap U_2)$ and, as before, there exists an open set $V_2\ni x_2$ such that $\overline{V_2}\subset(V_1\cap U_2)$. This way, we obtain a sequence $\{\overline{V_n}\}_n$ with the finite intersection property. Now, since $\overline{V_1}$ is compact and $\overline{V_n}\subset\overline{V_1}$ for each $n$, we have that $\bigcap_n \overline{V}_n$ is nonempty, and therefore there exists $x\in\bigcap_n\overline{V_n}\subset\big(\bigcap_nU_n\cap A\big)$, which proves that $\bigcap_n U_n$ is dense in $X$.     
		\end{itemize}
	\end{proof}
\end{thm}

Let us make some remarks about this proof. First of all, for BCT$1$ we make crucial use of the countability (IAN) of the space, which let us take those $B(x_n,r_n)$ with no major issues. In other words, the argument is valid because metric spaces verify the First Axiom of Countability. Do general uniform spaces (even complete ones) verify such property? Absolutely not. This indicates that there is very little chance of adapting the proof of BCT$1$ to complete uniform spaces.\\

On the other hand, to prove BCT$2$ we only required compacity for countable families, so if $X$ is countably compact and $T_3$ the proof is still valid. Now, since every uniform space is completely regular, it is also regular, and thus we have the following.

\begin{thm}
	Let $X$ be a uniform space. If it is countably compact, it is a Baire space.
\end{thm}

We just obtained our first condition in order to guarantee that a uniform space is Baire! It seems that BCT$2$ draws a better line to follow, so let us modify the hypothesis of countable compacity and let us ask the following: is a pseudocompact uniform space a Baire space? Let us begin by reminding the definition of a pseudocompact space.

\begin{definition}
	Let $X$ be a topological space. Then $X$ is said to be pseudocompact if its image under any continuous real function is bounded (in some texts it is also required that $X$ is completely regular).
\end{definition}

There is an immediate relation between pseudocompact spaces and countably compact spaces, but first \cite[3.10.20]{Eng89}:

\begin{prop}
	Every countably compact Tychonoff space is pseudocompact.
\end{prop}

We now propose ourselves to see if a pseudocompact uniform space is indeed a Baire space. We anticipate that the answer is affirmative, and we are going to prove it in three steps, starting by the following lemma, which comes from \cite[3.10.22]{Eng89} and may be found too in \cite[1.1.3]{pseudocompact}, but first we need a preliminary definition.

\begin{definition}
	\label{localfinite}
	A collection $\mathcal{A}$ of subsets of $X$ is said to be locally finite if for each point $x\in X$, there is an open neighborhood $V$ of $x$ such that $|\{A\in\mathcal{A}:A\cap V\}|<\infty$.
\end{definition}

\begin{lemma}
	\label{lemma1}
	Let $X$ be a completely regular space. The following conditions are equivalent:
	
	\begin{enumerate}
		\item $X$ is pseudocompact.
		
		\item If $\mathcal{A}=\{A_1,A_2,\ldots\}$ is a locally finite family of nonempty open subsets of $X$ then $\mathcal{A}$ is finite.
	\end{enumerate}
	
	\begin{proof}
		We will only prove that $1$ implies $2$. For the other implication we refer to \cite{Eng89} or \cite{pseudocompact}.
		
		Let us suppose that $2$ does not hold, and let $\mathcal{A}=\{A_1,A_2,\ldots\}$ is a local family of nonempty open subsets of $X$. For each $n$ we choose $x_n\in A_n$. Since $X$ is completely regular, for each $n$ there exists a map $f_n:X\rightarrow\R$ such that $f(x_n)=n$ and $f(X\setminus A_n)\subset\{0\}$. We define $f:X\rightarrow\R$ by $f(x)=f_1(x)+f_2(x)+\ldots$ for each $x\in X$.
		
		Since $\mathcal{A}$ is locally finite, we may define $f:X\rightarrow\R$ by $f(x)=f_1(x)+f_2(x)+\ldots$ since for each $x\in X$, $f(x)$ is a finite sum. We obtain a continuous function which is clearly not bounded. Hence, $X$ can not be pseudocompact.  
		
	\end{proof}
\end{lemma}

The next proposition comes from \cite[3.10.23]{Eng89} for Tychonoff spaces. It can also be found in \cite[1.1.5]{pseudocompact}. 

\begin{prop}
	\label{prop2}
	Let $X$ be a completely regular space. The following conditions are equivalent:
	
	\begin{enumerate}
		\item $X$ is pseudocompact.
		
		
		\item If $\{U_n\}_n$ is a decreasing sequence of nonempty open subsets verifying that $U_{n+1}\subset U_{n}$ for each $n$ then $\bigcap_n\overline{U_n}\neq\emptyset$.
	\end{enumerate}
	
	\begin{proof}
		We will only prove that $1$ implies $2$. For the other implication we refer to \cite{Eng89} or \cite{pseudocompact}.
		
		
		Let us suppose that $X$ is pseudocompact. From the previous lemma we infer that $\{U_n\}_n$ is not locally finite, and thus there exist $x\in X$ such that every open neighborhood of $x$ meets infinitely many $U_n$'s. Therefore, $x\in\bigcap_n\overline{U_n}$.
		
	\end{proof}
\end{prop}

Finally, we present the result below, which can be found in \cite[1.2.2]{pseudocompact} and follows immediately from the previous proposition.

\begin{thm}
	\label{pseudocompact}
	Let $X$ be a completely regular space. If $X$ is pseudocompact, then it is a Baire space.
	
	\begin{proof}        
		Let $\{U_1,U_2,\ldots\}$ be a countable family of open and dense subsets of $X$, and let $V\subset X$ be another open subset. We define $V_1=V\cap U_1$. Due to regularity, there exists an open set $V_2\subset U_2$ such that $\overline{V_2}\subset V_1$. This way, we may obtain a sequence of open sets $\{V_n\}_n$ such that $V_n\subset U_n$ and $\overline{V_{n+1}}\subset V_n$ for each $n$. This yields a sequence of open sets that verify the hypothesis of the previous proposition and thus $\bigcap_n\overline{V_n}\neq\emptyset$. Besides, $\overline{V_{n+1}}\subset V_n$, hence we have $\bigcap_n V_n\neq\emptyset$. Lastly, we observe that $\bigcap_n V_n\subset V\cap(\bigcap_n U_n)$, which concludes the proof.
		
	\end{proof}
\end{thm}

As we observe, in order to prove Theorem~\ref{pseudocompact}, we needed to introduce Definition~\ref{localfinite} and then Lemma~\ref{lemma1} followed by Prop~\ref{prop2}. However, it would be pretty clean to have a direct proof of Theorem~\ref{pseudocompact} without this previous definition and those additional result. We thus present an alternative proof:

\begin{proof}[Alternative proof of Theorem~\ref{pseudocompact}]
	
	Let $\{U_n\}_n$ be a countable family of dense open subsets of $X$, and let $A\subset X$ be another (nonempty) open subset. By regularity, there is a nonempty open subset $V_1\subset X$ such that $\overline{V_1}\subset U_1\cap A$. There is also nonempty open subset $V_2\subset X$ such that $\overline{V_2}\subset U_2\cap V_1\subset U_2\cap U_1\cap A$. This way, we obtain a sequence of nonempty open subsets $\{V_n\}_n$ verifying
	$$\overline{V_{n+1}}\subset V_n\cap U_{n+1}\subset A\cap\bigcap_{k=1}^{n+1} U_n$$
	\noindent
	for each $n$. Then
	$$\bigcup_{n\in\mathbb{N}}\overline{V_n}\subset A\cap\bigcap_{n\in\mathbb{N}}U_n.$$
	
	We just need to check that $\bigcup_{n\in\mathbb{N}}\overline{V_n}$ is nonempty.
	
	To do so, for each $n\in\mathbb{N}$ we take $x_n\in V_n$, and since $X$ is completely regular, we consider the continuous map $f_n:X\rightarrow[0,n]$ verifying that $f_n(x_n)=n$ and $f_n(X\setminus V_n)=\{0\}$.
	
	Let us suppose now that $\bigcap_{n\in\mathbb{N}}\overline{V_n}=\emptyset$, then for each $x\in X$ there is $m_x\in\mathbb{N}$ with $x\notin\overline{V_n}$ and thus $f_n(x)=0$ if $n\geq m_x$. We set $f:X\rightarrow[0,+\infty)$ defined by
	$$f(x):=\sum_{n=1}^\infty f_n(x).$$
	
	This sum will only have a finite amount of terms for each $x$. Also, $f$ is clearly unbounded since $f(x_n)\geq f_n(x_n)=n$. Therefore, we just need to verify that $f$ is continuous in order to obtain a contradiction with the fact that $X$ is pseudocompact. Given $x\in X$, we consider $W=X\setminus\overline{V_{m_x}}$, which clearly is an open neighborhood of $x$. Besides, for each $y\in W$ and each $n\geq m_x$ it is $y\in X\setminus V_n$ and hence $f_n(y)=0$. This implies that $f(y)=f_1(y)+\ldots+f_{{m_x}-1}(y)$ and thus $f\big|_W$ is continuous. We conclude that $f$ is continuous at $x$ and as a consequence, it is continuous in all of $X$.
	
\end{proof}

Now, since any uniform space is completely regular, we obtain the following.

\begin{corollary}
	Let $X$ be a uniform space. If $X$ is pseudocompact, then it is a Baire space.
\end{corollary}

To end this section, we introduce the notion of Čech-complete spaces.

\begin{definition}
	Let $X$ be a topological space. We say that a compact space $Y$ is a compactification of $X$ if there exists an embedding (a homeomorphism onto its image) $h:X\rightarrow Y$ such that $\overline{h(X)}=Y$. Most of the time, we will just say that $X\subset Y$ even though it is actually $X\simeq h(X)\subset Y$. In some texts it is also required for $Y$ to be Hausdorff.
\end{definition}

The Stone-Čech compactification
is defined as follows \cite[38.2]{munkres}:

\begin{thm}
	Let $X$ be a completely regular space. There exists a compactification $Y$ of $X$ such that every continuous map $f:X\rightarrow\R$ extends uniquely to a continuous map of $Y$ into $\R$.
\end{thm}

Furthermore, it can be proved that any two Stone-Čech compactifications are equivalent \cite[38.5]{munkres}. We observe that the Stone-Čech compactification is defined for completely-regular spaces, a property equivalent to uniformizability as we previously mentioned.

\begin{definition}
	A topological space $X$ is said to be Čech-complete if $X$ is completely regular and it is a $G_\delta$ in its Stone-Čech compactification.
\end{definition}

For instance, any complete metric space is Čech-complete. The reason we recalled this notion is because Čech-complete spaces are another example of Baire spaces \cite[3.9.3]{Eng89}.

\begin{thm}
	Let $X$ be a Čech-complete space. Then it is a Baire space.
\end{thm}

Therefore, since uniform spaces are already completely regular, a method we could use to see if a particular uniform space is Baire is studying its Stone-Čech compactification.

\section{Pseudocomplete spaces}

The arbitrary product of compact $T_2$ spaces is also compact $T_2$, and thus, a Baire space. However, the product of complete metric spaces does not even need to be metrizable, and we find a similar situation concerning Čech-complete spaces. Furthermore, the product of Baire spaces is not guaranteed to be a Baire space, even if it could possibly be one. We may consider the following question: is there a condition that, when satisfied, the arbitrary product of such spaces is guaranteed to be a Baire space? In order to answer it the notion of pseudocompleteness is introduced.\\

The definition of pseudobase was first introduced by Oxtoby in \cite{oxtoby} and can be found too in \cite{pseudo-product}.

\begin{definition}
	Let $X$ be a topological space, and let $\mathcal{B}$ be a family of nonempty open subsets of $X$. We say that $\mathcal{B}$ is a pseudobase if every nonempty open subset of $X$ contains some element of $\mathcal{B}$ as a subset. A pseudobase $\mathcal{B}$ is said to be locally countable if each member of $\mathcal{B}$ contains only a countable amount of members of $\mathcal{B}$.
\end{definition}

Subsequently, Todd came up with a modification of this notion. Instead, he left out the requirement for the sets of the sets of $\mathcal{B}$ to be open in order to obtain a pseudobase. This way, we have two definition of pseudocomplete spaces, the original definition set by Oxtoby in \cite{oxtoby} which appears also in \cite{pseudo-product} using open sets for the pseudobases, and the one proposed by Todd in \cite{pseudocomplete} which does not need open sets. One of them is clearly stronger than the other, however, it is still unknown if they are equivalent or not.
Here, we will be using the first of the two.

\begin{definition}
	Let $X$ be a topological space. We say that $X$ is quasiregular if every nonempty open subset of $X$ contains the closure of another nonempty open set. A topological space is said to be pseudocomplete if $X$ is quasiregular and there exists a countable family $\{\mathcal{B}_n\}_n$ of pseudobases verifying that if $U_n\in\mathcal{B}_n$ and $\overline{U_{n+1}}\subset U_n$, then $\bigcap_n U_n\neq\emptyset$.
\end{definition}

We observe that any complete metric space is pseudocomplete. It is clearly regular and for each $n\in\mathbb{N}$ we can consider the pseudobase $\mathcal{B}_n=\{B(x,r):x\in X,r<\frac{1}{n}\}$. Similarly, every locally compact Hausdorff space is pseudocomplete, we just need to take for each $\mathcal{B}_n$ the class of all open subsets whose closure is compact. It can even be proved that every Čech-complete space is pseudocomplete too \cite{pseudo-product}.

Let us see now that every pseudocomplete space verifies the Baire condition.

\begin{thm}[\cite{oxtoby}]
	If $X$ is a pseudocomplete space, then it is a Baire space.
	
	\begin{proof}
		Let $\mathcal{A}$ be a countable family of open and dense subsets of $X$, and let $V\subset X$ be another open subset. Since $X$ is quasiregular, there exists $B_0\in\mathcal{B}_0$ such that $\overline{B_0}\subset V\cap U_0$. We continue this process: for an already chosen $B_n$ from $\mathcal{B}_n$ we take $B_{n+1}\in\mathcal{B}_{n+1}$ such that $\overline{B_{n+1}}\subset V\cap U_{n+1}$. Since $X$ is pseudocomplete, it follows that $\bigcap_n B_n\neq \emptyset$, which concludes the proof.
		
	\end{proof}
\end{thm}

By analyzing the proof we realize that the definitions of pseudoregular and pseudobase capture the minimum conditions so that the space is Baire, at least by following the classical proof of BCT.

Now there are some other results we should consider.

\begin{thm}[\cite{oxtoby}]
	Let $Y$ be a quasiregular space equipped with a pseudobase $\mathcal{B}$ consisting of those subsets whose closure is countably-compact. Then any $G_\delta$ subbspace dense in $Y$ is pseudocomplete, and thus, Baire. 
\end{thm}

The following can be found explicitly in \cite{pseudocomplete} but the argument actually comes from \cite{pseudo-product}.

\begin{theorem}
	\label{pscmlch}
	
	If a topological space is completely metrizable topological or locally compact and Hausdorff then it is pseudocomplete.
\end{theorem}

\begin{thm}[\cite{oxtoby}]
	The arbitrary product of any family of pseudocomplete spaces (equipped with the product topology) is pseudocomplete.
\end{thm}

Therefore, it follows immediately that:

\begin{corollary}
	The product of pseudocomplete spaces is a Baire space.
\end{corollary}

\begin{corollary}
	The arbitrary product of complete metrizable spaces is a Baire space, even if it is not metrizable.
\end{corollary}

We return now to uniform spaces, and try to find some conditions under which a general uniform space is pseudocomplete. Todd says in \cite{pseudocomplete} that regularity implies quasiregularity, though he does not prove it. We rather present a proof.

\begin{prop}
	Let $X$ be a regular space. Then $X$ is quasiregular.
	
	\begin{proof}
		Let $\emptyset\neq U\subset X$ be an open subset, and let us verify that it contains the closure of another nonempty open set.
		
		We take $x\in U$, and set $F=X\setminus U$, which is closed. Since $X$ is regular, there exist two open disjoint neighborhoods $U_x,U_F$ of $x$ and $F$, respectively. 
		
		Clearly, $U_x\subset U$ because $U_x\cap (X\setminus U)\subset U_x\cap U_F=\emptyset$. Now let $z\in \overline{U}_x$ and let us check that $z\in U$. If $z\notin U$ then $z\in X\setminus U=F\subset U_F$, so $U_F$ is an open neighborhood of $z$ and thus $U_x\cap U_F\neq\emptyset$, which cannot happen. Therefore, we conclude that $\overline{U}_x\subset U$.
		
	\end{proof}
\end{prop}

Let us take now a general uniform space $(X,\U)$. Since $X$ is completely regular, it is also regular, and by the previous proposition, it is quasiregular. Therefore, in order to obtain a Baire space the only remaining thing is obtaining a pseudobase.\\

Given a uniform space $(X,\U)$, if $\mathcal{B}$ is a base of uniformity, then it follows that $\mathcal{B}^*=\{U[x]:x\in X,U\in\mathcal{B}\}$ is a pseudobase of the corresponding topological space $(X,\mathcal{T_U})$. Our desire would be to obtain a family of bases of uniformity whose corresponding families of pseudobases verify the conditions required by the space in order to be pseudocomplete.\\

For instance, let us suppose that $\mathcal{U}$ admits a countable base of uniformity $\mathcal{B}=\{U_0,U_1,\ldots\}$. We define $B_n^*=\{U[x]:x\in X,U\in\mathcal{U},U[x]\times U[x]\subset U_n,\}$. 

Given an open set $V\subset X$ and given $x\in V$, there exists $U\in\mathcal{U}$ such that $U[x]\subset V$. We define $W=U\cap U_n\in\mathcal{U}$, and we see that $W[x]\subset U[x]\subset V$ and $W[x]\times W[x]\subset U_n$, so $W[x]\in B_n^*$, which proves that $B_n^*$ is indeed a pseudobase. Therefore, we obtain a countable family of pseudobases $\{B_n^*\}_n$.

Given a sequence $\{V_n[x_n]\}_n$ such that $V_n[x_n]\in B_n^*$ and $\overline{V_{n+1}[x_{n+}]}\subset V_n[x_n]$ for each $n$, we observe that for each $U\in\mathcal{U}$ there exists $U_k\subset U$ and the corresponding $V_k[x_k]\in B_k^*$ verifies $V_k[x_k]\times V_k[x_k]\subset U_k\subset U$. Besides, $\overline{V_m[x_m]}\subset V_k[x_k]$ if $m>k$, so the sequence $\{x_n\}_n$ is a Cauchy sequence. Therefore, provided that the space is complete, we may confirm that the limit $x=\lim_n x_n$ exists, and because $x\in V_n[x_n]$ for each $n$, it is concluded that $x\in\bigcap_n V_n[x_n]$.\\

Let us assume now that $(X,\U)$ is totally bounded, and let $B$ be a base of uniformity. For each $n$ we take $B_n=\mathcal{T_U}$. Let $\{V_n\}_n$ be a sequence where $V_n\in B_n$ and $\overline{V_{n+1}}\subset V_n$. Since each $V_n$ is nonempty, by making use of the Axiom of Choice, we may take $x_n\in V_n$, which generates a sequence $\{x_n\}_n$. We assumed that $X$ is totally bounded, so each net contains a Cauchy subnet \cite[39.8]{willard}, and, in particular, $\{x_n\}_n$ contains a Cauchy subnet $\{x_{n_k}\}_k$. Again, provided that the space is complete, there exists the limit $x=\lim_k x_{n_k}$. These $\{x_{n_k}\}_k$ match some $V_{n_k}$, for which $x\in V_{n_k}$ for each $k$. Now, for each $V_n$ there exists some $V_{n_k}\subset V_n$, so $x\in V_n$ for each $n$ and thus $\bigcap_n V_n\neq\emptyset$.\\

We have proved the following.

\begin{thm}
	\label{psunif}
	
	Let $(X,\U)$ be a complete uniform space. If one of the following conditions is satisfied, then it is a Baire space:
	
	\begin{enumerate}
		\item $(X,\U)$ admits a countable base of uniformity.
		
		\item $(X,\U)$ is totally bounded.
	\end{enumerate}
\end{thm}

One the one hand, as Theorem~\ref{metrization} says, if $(X,\U)$ is a separated uniform space which is complete and admits a countable base of uniformity then it is completely metrizable, and thus verifies BCT$1$.

On the other hand, if $(X,\U)$ is a separated uniform space which is complete and totally bounded, then by Theorem~\ref{compacto} it is compact (and Hausdorff), and thus verifies BCT$2$.

Moreover, Theorem~\ref{psunif} verifies the  hypothesis of Theorem~\ref{pscmlch}.\\

We observe that, even if we proved something which was already stated, we did not only provide a different and more constructive way to do so, but we have just illustrated a path to take in order to prove different hypothesis for a uniform space to be Baire. Therefore, we open up a variety of possibilities by encouraging to construct an appropriate family of pseudobases.

\end{document}